\documentclass[12pt,oneside,openany, article]{memoir}

\nouppercaseheads

\usepackage[amsthm,thmmarks,hyperref]{ntheorem}

\usepackage[noTeX]{mmap}

\usepackage[english.us]{babel}
\usepackage[leqno]{mathtools}
\usepackage{empheq}

\usepackage{caption}

\usepackage[scr]{rsfso}
\usepackage{upgreek}

\usepackage[compress]{cite}
\usepackage{hypernat} 

\usepackage{fix-cm}
\usepackage[cm]{sfmath}
\usepackage{sansmathaccent}

\usepackage{microtype}
\newtheorem{theorem}{Theorem}[chapter]
\newtheorem{proposition}[theorem]{Proposition}

\newtheorem{corollary}[theorem]{Corollary}

\theoremstyle{remark}
\newtheorem{remark}[theorem]{Remark}

\usepackage{subfig}
\usepackage{wrapfig}
\usepackage{array}

\usepackage{graphicx,xcolor}
\usepackage{eso-pic}
\usepackage[absolute,overlay]{textpos}

\usepackage{pdfpages}
\definecolor{refkey}{rgb}{0,0,1}
\definecolor{labelkey}{rgb}{1,0,0}

\usepackage{todonotes}

\usepackage{xr-hyper}
\usepackage{hyperxmp}
\usepackage{zref, nameref}
\usepackage{url}
\usepackage[pdftex,bookmarks,pdfnewwindow,plainpages=false,unicode,pdfencoding=auto]{hyperref}
\newenvironment{phantomequation}[1][]{\refstepcounter{equation}}{}

\usepackage{bookmark}
\usepackage[verbose]{placeins}
\usepackage{enumitem}
\usepackage{tikz}
\usepackage{pgfplots}

\usetikzlibrary{arrows, patterns,
decorations.pathreplacing, decorations.pathmorphing, shapes.symbols, spy}

\usepackage{relsize}

\usepackage{amssymb} 

\setsecheadstyle{\large\bfseries\raggedright}
\setsubsecheadstyle{\normalsize\bfseries\raggedright}

\newenvironment{claim}[1][{\textup{(\theequation)}}]{\refstepcounter{equation}\vglue10pt
\begin{trivlist}
\item[{\hskip\labelsep#1}]}{\vglue10pt\end{trivlist}}
\hypersetup{
colorlinks=true,
linkcolor=black,
citecolor=black,
urlcolor=blue,
pdfauthor={Victor Ivrii},
pdfauthortitle={Professor, Department of Mathematics, University of Toronto},
pdftitle={Pointwise Spectral Asymptotics out of the Diagonal near Degeneration},
pdfsubject={Sharp Spectral Asymptotics},
pdfkeywords={Microlocal Analysis,  Sharp Spectral Asymptotics},
pdfcopyright={Copyright (C) 2019 Victor Ivrii},
pdfcontactemail={ivrii@math.toronto.edu},
pdfcontacturl={http://www.math.toronto.edu/ivrii},
pdfcontactaddress={40 St. George St.},
pdfcontactcity={Toronto, ON},
pdfcontactpostcode={M5S 2E4},
pdfcontactcountry={Canada},
pdflang={en},
baseurl={http://www.math.toronto.edu/ivrii/monsterbook.pdf/},
pdflicenseurl={http://creativecommons.org/licenses/by-nc-nd/3.0/},
bookmarksdepth={4},
pdfpagelayout=SinglePage
}

\numberwithin{equation}{chapter}

\newcommand{\sC}{\mathscr{C}}

\newcommand{\bR}{\mathbb{R}}
\newcommand{\bC}{\mathbb{C}}
\newcommand{\W}{\mathsf{W}}
\newcommand{\T}{\mathsf{T}}

\newcommand{\N}{\mathsf{N}}

\newcommand{\corr}{\mathsf{corr}}

\newcommand{\Long}{\mathsf{long}}
\newcommand{\short}{\mathsf{short}}

\renewcommand{\Im}{\operatorname{Im}}

\newcommand{\Ai}{\operatorname{Ai}}
\newcommand{\Hess}{\operatorname{Hess}}

\newcommand{\x}{\mathsf{x}}

\newcommand{\eff}{{\nu}}

\title{Pointwise Spectral Asymptotics out of the Diagonal near Degeneration }

\author{Victor Ivrii\thanks{This research was supported in part by National Science and Engineering  Research Council (Canada) Discovery Grant  RGPIN 13827}}

\begin{document}
\maketitle

\begin{abstract}
We establish uniform (with respect to $x$, $y$) semiclassical asymptotics and estimates for the Schwartz kernel $e_h(x,y;\tau)$ of spectral projector for a second order elliptic operator inside domain under microhyperbolicity (but not $\xi$-microhyperbolicity) assumption. While such asymptotics for  its restriction to the diagonal $e_h(x,x,\tau)$ and, especially, for its trace $\N_h(\tau)= \int e_h(x,x,\tau)\,dx$ are well-known, the out-of-diagonal asymptotics are much less explored, especially uniform ones.

Our main tools: microlocal methods, improved successive approximations and geometric optics methods.
\end{abstract}

\chapter{Introduction}
\label{sect-1}

In this paper we consider a self-adjoint scalar operator  which is elliptic second order differential operator with $\sC^K$-coefficients ($K=K(d,\delta ))$ 
\begin{gather}
A=\sum_{j,k} \bigl(hD_j-V_j(x)\bigr) g^{jk}(x)\bigl(hD_k-V_k(x)\bigr)+V(x),
\qquad g^{jk}=g^{kj}, 
\label{eqn-1.1}\\
\intertext{and under microhyperbolicity condition}
|V(x)-\tau| +|\nabla V(x)|\ge \epsilon _0 \qquad \forall x\in B(0,1)\subset \bR^d, \ d\ge 2,
\label{eqn-1.2}
\end{gather}
we establish sharp (with $O(h^{1-d})$ remainder) asymptotics of the Schwartz kernel of the spectral projector $e_h(x,y,\tau)$ as $h\to +0$. 

Under $\xi$-microhyperbolicity assumption 
\begin{gather}
|V(x)-\tau| \ge \epsilon _0 \qquad \forall x\in B(0,1)\subset \bR^d
\label{eqn-1.3}
\end{gather}
such asymptotics of $e_h(x,x,\tau)$ is well-known for much more general operators including matrix ones (with a proper definition of $\xi$-microhyperbolicity).  However out-of-diagonal asymptotics, especially uniform with respect to $x,y$, are much less explored and for simple main part and sharp remainder estimate in addition to $\xi$-microhyperbolicity 
\begin{gather}
|a(x,\xi) -\tau|+|\nabla_\xi a(x,\xi)|\ge \epsilon_0
\label{eqn-1.4}
\end{gather}
 where $a(x,\xi)$ is the principal symbol) require strong convexity of the energy surface $\{\xi\colon a(x,\xi)=\tau\}$; see \cite{OOD}, Theorem~\ref{OOD-thm-1.1}.

On the other hand,   under only microhyperbolicity condition 
\begin{gather}
|a(x,\xi) -\tau|+\nabla_{x,\xi} a(x,\xi)|\ge \epsilon_0
\label{eqn-1.5}
\intertext{(which could be generalized for matrix operators) sharp asymptotics of}
\N_h(\tau)=\int e_h(x,x,\tau)\,dx
\label{eqn-1.6}
\end{gather}
 are well-known (see \cite{monsterbook}, Chapter \ref{monsterbook-sect-4})\,\footnote{\label{foot-1}  And even near the boundary  (see \cite{monsterbook}, Chapter \ref{monsterbook-sect-7}).}. However in this case  even uniform asymptotics of 
 $e_h(x,x,\tau)$ are much less explored and only for operator (\ref{eqn-1.1}) (see  Subsection~\ref{monsterbook-sect-5-3-1},   Subsubsection~``\nameref{monsterbook-sect-5-3-1-3}'')  because while microhyperbolicity conditions does not allow \emph{short periodic trajectories}, it allows \emph{short loops}\footnote{\label{foot-2} But $\xi$-microhyperbolicity does not allow short loops.}.  

\begin{theorem}\label{thm-1.1}
Consider elliptic self-adjoint operator \textup{(\ref{eqn-1.1})} with smooth coefficients. Assume that microhyperbolicity condition \textup{(\ref{eqn-1.2})}
is fulfilled.
Then asymptotics 
\begin{gather}
e_h(x,y,\tau)=e_h^\W (x,y,\tau) +\left\{\begin{aligned}
&O(h^{1-d}) \quad && d\ge 3,\\
&O(h^{-\frac{4}{3}}) && d=2
\end{aligned}\right.
\label{eqn-1.7}\\
\intertext{holds for all   $x,y\in B(0,\epsilon )$  with}
e_h^\W (x,y;\tau)=(2\pi h)^{-d}  \int _{\{a(\frac{1}{2}(x+y), \xi)<\tau\}} e^{ih^{-1}\langle x-y,\xi\rangle}\,d\xi,
\label{eqn-1.8}\\
a(x,\xi)= \sum_{j,k} g^{jk}(x)\xi_j\xi_k +V(x).
\label{eqn-1.9}
\end{gather}
\end{theorem}

\begin{theorem}\label{thm-1.2}
Consider elliptic self-adjoint operator \textup{(\ref{eqn-1.1})} with smooth coefficients. Assume that microhyperbolicity condition \textup{(\ref{eqn-1.2})}
is fulfilled. Let $d=2$ and $x,y\in B(0,\frac{1}{2})$.
\begin{enumerate}[wide,label=(\roman*),  labelindent=0pt]
\item
Asymptotics
\begin{gather}
e_h(x,y,\tau)=e_h^\W (x,y,\tau) + O(h^{-1})
\label{eqn-1.10}
\end{gather}
holds if \underline{either} condition \textup{(\ref{eqn-1.3})} is fulfilled, \underline{or} $\ell(x,y) \ge h^{\frac{1}{3}}$,
\underline{or} $\ell(x,y) \ge h^{\frac{2}{5}}$ and
\begin{multline}
\bigl|   4V(x-\tau)V(y-\tau) -|x-y|^2_{(g_{jk})} |\nabla V|_{(g^{jk})}^2 +\langle x-y, \nabla V\rangle^2 \bigr|\\
\ge \epsilon \ell(x,y)^2 
\label{eqn-1.11}
\end{multline}
where
\begin{gather}
\ell(x,y) =\max (|x-y|,\, |V(x)-\tau|,\, |V(y)-\tau|,\, h^{\frac{2}{3}}),
\label{eqn-1.12}
\end{gather}
$(g_{jk})=(g_{jk})^{-1}$ and $\nabla V$ is calculated at $\frac{1}{2}(x+y)$.

\item
Without condition \textup{(\ref{eqn-1.11})}  for $h^{\frac{1}{3}}\ge \ell(x,y)\ge h^{\frac{2}{5}}$ the following estimate holds:
\begin{gather}
e_h(x,y,\tau)=O(h^{-\frac{2}{3}}\ell (x,y)^{-1}).
\label{eqn-1.13}
\end{gather}

\item
Further, for $\ell(x,y)\le h^{\frac{2}{5}}$   asymptotics 
\begin{gather}
e_h(x,y,\tau)=e_h^\W (x,y,\tau) +e_{\corr, k, h} (x,y,\tau)+ O(h^{-1}) 
\label{eqn-1.14}
\end{gather}
holds   with  correction term $e_{\corr, k, h} (x,y,\tau)$ to be defined by \textup{(\ref{eqn-2.46})}   later  if
\underline{either}   condition \textup{(\ref{eqn-1.11})}  is fulfilled \underline{or} $\ell (x,y)\le h^{\frac{4}{9}}$.

\item
Without condition \textup{(\ref{eqn-1.11})}  for $h^{\frac{2}{5}}\ge \ell(x,y)\ge h^{\frac{4}{9}}$ the following estimate holds:
\begin{gather}
e_h(x,y,\tau)= e_{\corr, k, h} (x,y,\tau)+ O( h^{-\frac{5}{3}}\ell(x,y)^{\frac{3}{2}}) .
\label{eqn-1.15}
\end{gather}
\item
Finally, as $x=y$ one can replace $k$ by $0$ while preserving remainder estimate  and define  $e_{\corr, 0, h} (x,y,\tau)$  as
\begin{gather}
2h^{-\frac{8}{3}}  \int_{-\infty}^\tau (\tau-\tau') \Ai (-2h^{-\frac{2}{3}} ( \tau' - V(x))\,d\tau' - \frac{1}{2\pi h^2 } (\tau-V(x))_+
\label{eqn-1.16}
\end{gather}
provided $|\nabla V|_{g^{jk}}=1$ where $\Ai(\cdot)$ denotes Airy function.
\end{enumerate}
\end{theorem}

\begin{remark}\label{remark-1.3}
As $d=1,2$ and $x=y$ asymptotics are already   known:   see \cite{monsterbook}, Subsection~\ref{monsterbook-sect-5-3-1},   Subsubsection~\nameref{monsterbook-sect-5-3-1-4}).
\end{remark}

\paragraph{Plan of the paper.}
\emph{General idea.} Our main method is to prove (as $d\ge 2$) \emph{Tauberian asymptotics}
\begin{gather}
e_h(x,y,\tau) = e^\T_ {T,h} (x,y,\tau)+ O(h^{1-d})
\label{eqn-1.17}\\
\intertext{with the  \emph{Tauberian expression}}
e^\T_ {T,h} (x,y,\tau)=\frac{1}{h}\int_{-\infty}^\tau F_{t\to h^{-1}\tau} \Bigl(\bar{\chi}_T(t)u_h(x,y,t)\Bigr)\,d\tau 
\label{eqn-1.18}
\end{gather}
where $u_h(x,y,t)$ is the \emph{propagator} (the Schwartz kernel of $e^{ih^{-1}t A_h}$), and  $\bar{\chi}_T(t)=\bar{\chi}(T^{-1}t)$ is the appropriate cut-off. Then from Tauberian expression we pass to Weyl expression and a correction term if needed. So far our arguments are similar to those of \cite{OOD} where we considered points near boundary but assumed $\xi$-microhyperbolicity. 

To prove Tauberian asymptotics  we need to prove that Fourier transform in (\ref{eqn-1.18})  with $T\asymp 1$ is $O(h^{1-d})$.

 \emph{Section 2.}  We start from the toy-model (\ref{eqn-2.1}), study Hamiltonian trajectories and then do all calculations explicitly. We arrive to the oscillatory integral which has four stationary points, corresponding to  two trajectories from $x$ to $y$ and two trajectories from $y$ to $x$ on the energy level $\tau$, provided we are in the \emph{regular zone} where  the left-hand expression in (\ref{eqn-1.11}) without absolute value is positive: two are short trajectories and two are long ones. If the same expression is negative (\emph{shadow zone}) there are neither any trajectories nor stationary points, and if (\ref{eqn-1.11}) fails (\emph{singular zone}) these points almost coincide.

Then we repeat these steps for the generalized toy-model (\ref{eqn-2.31}).

\emph{Section 3.} We follow the same path, but we get semi-explicit expression: using special coordinates and $x_1$-microhyperbolicity we construct in the standard way 
$F_{x_1\to h^{-1}\xi_1,y_1\to- h^{-1}\eta_1} u_h(x,y,t)$ as an oscillatory integral and then make inverse Fourier transform. Then we prove (\ref{eqn-1.17}) and either pass from Tauberian expression to Weyl expression (as $\ell(x,y) \ge h^{\frac{2}{5}}$) and estimate the error, or use the perturbation method (as $\ell(x,y)\le h^{\frac{2}{5}}$)
and estimate
\begin{gather*}
\bigl(e^\T _{T,h} (x,y,\tau) - \bar{e}^\T _{T,k, h} (x,y,\tau) \bigr)- \bigl(e^\W_h (x,y,\tau)- \bar{e}^\W_{k,h} (x,y,\tau)\bigr)
\end{gather*}
where $\bar{e}^\T _{T,k, h} (x,y,\tau)$ and $\bar{e}^\W_{k,h} (x,y,\tau)$ correspond to an appropriate generalized toy-model. Then we define correction term as
\begin{gather*}
e_{\corr, k,h} = \bigl(\bar{e}^\T _{T,k, h} (x,y,\tau) - \bar{e}^\W _{k, h} (x,y,\tau) \bigr).
\end{gather*}

\emph{Section 4.}
Here we show how simple rescaling arguments allows to get rid off assumption (\ref{eqn-1.3}). We also discuss easier case $d=1$.

\chapter{Toy-model}
\label{sect-2}

Let us consider a toy-model operator with $d\ge 2$, $g^{jk}=\updelta_{jk}$, $V_j(x)=0$ and $V(x)=-x_1$:
\begin{gather}
\bar{A}_h \coloneqq \frac{1}{2}h^2D^2 - x_1 =\frac{1}{2} h^2D_1^2 + \frac{1}{2} h^2D^{\prime\,2} -x_1
\label{eqn-2.1}
\end{gather}
 By shift $x_1\mapsto x_1 + \tau$ we can reduce $\tau\in \bR$ to $\tau=0$. After this by rescaling $x\mapsto h^{\frac{2}{3}}$, $\tau\mapsto h^{\frac{2}{3}}$ we can reduce $h>0$ to $h=1$.
 
\section{Hamiltonian trajectories}
\label{sect-2.1}

\begin{proposition}\label{prop-2.1}
Consider toy-model operator \textup{(\ref{eqn-2.1})}. Then
\begin{enumerate}[wide,label=(\roman*),  labelindent=0pt]
\item
Hamiltonian trajectories from $(\bar{x},\bar{\xi})=(\bar{x},\bar{\xi}_1,\bar{\xi}')$, 
$\bar{x}_1>0$ on the energy level $0$ (so, $\bar{\xi}_1 =\mp \sqrt{2\bar{x}_1 -|\bar{\xi}'|^2}$) are 
\begin{gather*}
\xi '=\bar{\xi}', \ x_1= \bar{x}_1+\frac{ t^2}{2} \mp t\sqrt{2\bar{x}_1 -|\bar{\xi}'|^2},\ x'=\bar{x}'+\bar{\xi}' t ,\ \xi_1=t\mp \sqrt{2\bar{x}_1-|\bar{\xi}'| ^2}.
\end{gather*}
\item
If $d=2$ and $\pm \bar{\xi}_1< 0$ their projections to $x$-space are parabolas
\begin{gather*}
x_1 = \frac{1}{2\bar{\xi}_2^2} \Bigl(x_2-\bar{x}_2 \mp \bar{\xi}_2\sqrt{2\bar{x}_1-\bar{\xi}_2^2}\Bigr)^2+\frac{\bar{\xi}_2^2}{2}
\end{gather*}
 which at  ${t=\pm \frac{2\bar{x}_1}{\sqrt{2\bar{x}_1- \bar{\xi}_2 ^2}}}$  are tangent to the parabola
 $\Gamma =\{x_1=\frac{1}{4\bar{x}_1 }|x'-\bar{x}_2|^2\}$.

\item
Any point $x$ above $\Gamma$ (that is with $x_1>\frac{1}{4\bar{x}_1 }(x_2-\bar{x}_2^2$) is covered by two such rays, one of then touches $\Gamma$ between $\bar{x}$ and  $x$, and another outside of this segment. Any point $x$ below $\Gamma$ (that is with $x_1<\frac{1}{4\bar{x}_1 }(x_2-\bar{x}_2)^2$) is not reachable from $\bar{x}$ by rays on the energy level $0$, and any point on $\Gamma$ is reachable by just one ray (when point above $\Gamma$ tends to $\Gamma$ both rays tend to the same limit).

\item
The vertices of these parabolas are on the ellipse \\
$L=\{(2x_1-\bar{x}_1)^2 + (x_2-\bar{x}_2)^2 =\bar{x}_1^2\}$; between $\Gamma $ and $L$ along both rays $\xi_1$ have the same sign (opposite to one of $\bar{xi}_1$) while inside $L$ (brown on Figure~\ref{fig-1}) for long rays this is true only for a long ray, while for a short ray sign of $\xi_1$ coincides with the sign of $\bar{\xi}_1$. \end{enumerate}
\end{proposition}

\begin{figure}[h!]
\centering
\begin{tikzpicture}[xscale=8, yscale=8]
\fill[brown!20] (1,0) arc (0:180:.5) ;

\draw[thick, ->] (0,0)--(0,1) node[right]  {$x_2$};
\clip (-.1,-.1) rectangle (1.1,1.1);
\foreach \c in {0,.1,...,1,1}
{\draw[domain=0:2,samples=100]  plot ({1+\x*\x -2*\x*sqrt(1-\c*\c)}, {\c*\x});
}

\foreach \c in {0,.1,...,1}
{\draw[domain=0:2,samples=100]  plot ({1+\x*\x +2*\x*sqrt(1-\c*\c)}, {\c*\x});
}

\foreach \c in {.6}
\draw[red, thick, domain=0:2,samples=100]  plot ({1+\x*\x -2*\x*sqrt(1-\c*\c)}, {\c*\x});

\foreach \c in {.8}
\draw[blue, thick, domain=0:2,samples=100]  plot ({1+\x*\x -2*\x*sqrt(1-\c*\c)}, {\c*\x});

\draw[magenta, domain=0:2,samples=100, thick] plot ({\x^2}, {\x});

\end{tikzpicture}
\caption{\label{fig-1} $d=2$. Hamiltonian rays from $\bar{x}$  as $\bar{\xi}_2>0$, $t>0$, scaled.}
\end{figure}
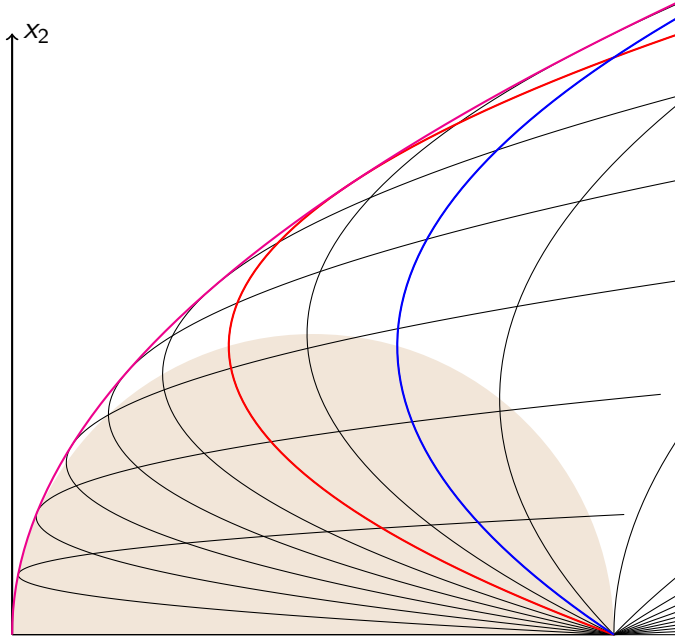

\begin{proof}
Easy proof by direct calculations is left to the reader.
\end{proof}

\begin{remark}\label{remark-2.2}
Observe that for a trajectory between points $x$ and $y$
\begin{enumerate}[wide,label=(\roman*),  labelindent=0pt]
\item
For a long trajectory $\ell(x,y)\asymp |x'-y'| +\eff(x) +\eff(y)$ with $\eff(x)= |V(x)-\tau| $ and $T(x,y)\asymp \ell^{\frac{1}{2}}(x,y)$ is a time to reach.

\item
\underline{Either} $\ell(x,y)\asymp \ell^0(x,y)\coloneqq |x-y|$ and $T^0(x,y)\asymp T(x,y)$ where $T^0(x,y)$ is time to reach along short trajectory,
\underline{or} $|\xi_{\Long} -\xi_{\short}| \asymp \eff(x)^{\frac{1}{2}} $ and $|\eta_{\Long} -\eta_{\short}| \asymp \eff(y)^{\frac{1}{2}} $, where $\xi_*$ and $\eta_*$ correspond to these trajectories at $x$ and $y$ for $|t|\asymp T(x,y)$ or $T^0(x,y)$.

\item
Further, since $\xi'_{\short} \ell^0 = \xi'_{\Long} \ell$, we conclude  that
\begin{gather}
\ell^0<\epsilon \ell\implies |\xi'_{\short}| > \epsilon^{-1}|\xi'_{\Long}|.
\label{eqn-2.2}
\end{gather}
\end{enumerate}
\end{remark}

\section{Calculations}
\label{sect-2.2}

As usual, consider $u^\pm_h(x,y,t)= u_h(x,y,t)\uptheta(\pm t)$ with  Heaviside function $\uptheta$; then
\begin{gather}
\Bigl( hD_t -\frac{1}{2}  h^2 D_1^2 -\frac{1}{2} h^2 |D'|^2+ x_1\Bigr) u^\pm _h = \mp i h\updelta(t)\updelta(x-y). \label{eqn-2.3}\\
\intertext{making $h$-Fourier transform}
\widehat{u_h^\pm} (\xi,y,\tau) \coloneqq F_{t\to h^{-1}\tau , x\to h^{-1}\xi }u^\pm _h= \iint e^{-ih^{-1}(t\tau +\langle x,\xi\rangle)} u^\pm_h(x,y,t)\,dxdt
\notag\\
\intertext{we get as $\mp \Im \tau >0$}
\Bigl(\tau - \frac{1}{2} \xi_1^2 - \frac{1}{2} |\xi'|^2 + ih \partial _{\xi_1}\Bigr) \widehat{u_h^\pm} = \mp i  e^{ih^{-1}(-y_1\xi_1-\langle y',\xi'\rangle)}   ;
\label{eqn-2.4}\\[7pt]
\intertext{rewriting it as}
\partial_{\xi_1} \Bigl(e^{-ih^{-1}((\tau-|\xi'|^2 /2 ) \xi_1-\xi_1^3/6)} \widehat{u_h^\pm} \Bigr)= \mp e^{ih^{-1}(-y_1\xi_1-\langle y',\xi'\rangle - (\tau-|\xi'|^2/2) \xi_1 +\xi_1^3/6)},
\notag
\end{gather}
we arrive to
\begin{multline}
 \widehat{u_h^\pm} (\xi,y,\tau)\\
 =  \mp  e^{ih^{-1}((\tau-|\xi'|^2/2) \xi_1-\xi_1^3/6)}\int _{\mp \infty}^{\xi_1} e^{ih^{-1}(-y_1\eta_1-\langle y',\xi'\rangle - (\tau-|\xi'|^2/2)\eta_1+\eta_1^3/6)}\,d\eta_1
 \label{eqn-2.5}
 \end{multline}
 and making partial inverse  $h$-Fourier transform
 \begin{multline}
 F_{t\to h^{-1}\tau} u_h^\pm 
 = \mp (2\pi h)^{-d}\\
 \iint \Bigl(\int _{\mp \infty}^{\xi_1} e^{ih^{-1}(x_1\xi_1-y_1\eta_1+\langle x'-y', \xi'\rangle  +( \tau -|\xi'|^2/2) (\xi_1-\eta_1) -(\xi_1^3-\eta_1^3)/6))}\,d\eta_1\Bigr)\,  d\xi_1d\xi'.
 \label{eqn-2.6}
\end{multline}
Then as $\tau \in \bR$
\begin{multline}
 F_{t\to h^{-1}\tau} u_h  
 = (2\pi h)^{-d}\times \\
 \iint \Bigl(\int _{- \infty}^{\infty} e^{ih^{-1}(x_1\xi_1-y_1\eta_1+\langle x'-y', \xi'\rangle + ( \tau -|\xi'|^2/2 ) (\xi_1-\eta_1) -(\xi_1^3-\eta_1^3)/6))}\,d\eta_1\Bigr)
\, d\xi_1d\xi'  \\
= h^{(1-2d)/3 } J ( h^{-\frac{2}{3}}\tau ,h^{-\frac{2}{3}} x_1, h^{-\frac{2}{3}} y_1,h^{-\frac{2}{3}} (x'-y')),
\label{eqn-2.7}
\end{multline}
with
\begin{multline}
J (\tau   ,   x_1,    y_1,  z')\coloneqq   (2\pi  )^{-d}\times \\
\iiint   e^{i (x_1\xi_1-y_1\eta_1+\langle z', \xi'\rangle  +( \tau - |\xi'|^2/2) (\xi_1-\eta_1) -(\xi_1^3-\eta_1^3)/6))}\,
\,d\xi_1 d\eta_1 d\xi'.
\label{eqn-2.8}
\end{multline}
We can see easily that the phase function 
\begin{multline}
\Phi (x_1,y_1,z'; \xi_1,\eta_1,\xi';\tau) \\
= x_1\xi_1-y_1\eta_1+\langle z', \xi'\rangle  +( \tau -\frac{1}{2}|\xi'|^2) (\xi_1-\eta_1) -\frac{1}{6}(\xi_1^3-\eta_1^3))
\label{eqn-2.9}
\end{multline}
has four stationary points (with respect to $(\xi_1,\eta_1,\xi')$): on all of them
\begin{gather}
2x_1=-2\tau +\xi_1^2+|\xi'|^2,\quad 2y_1=-2\tau +\eta_1^2+|\xi'|^2,\quad z'=  \xi' (\xi_1-\eta_1),
\label{eqn-2.10}\\
\intertext{and for corresponding trajectories  on the energy $\tau$}
T(x,y)= |\xi_1-\eta_1|
\label{eqn-2.11}
\end{gather}
but  
\begin{enumerate}[label=(\alph*)]
\item
on two of them $|\xi_1-\eta_1|$ is the smallest (short rays)   and \\
  $\ell^0(x,y)  \asymp |x_1-y_1|\asymp |x-y| $ with   $x_1\ge 0$, $y_1\ge 0$.
  \item
on the other two $|\xi_1-\eta_1|$ are largest  (long rays) $\xi_1$ and $\eta_1$ have opposite signs  and\\   $\ell (x,y) \asymp x_1+y_1\asymp |x-y|+ x_1+y_1$ with $x_1\ge 0$, $y_1\ge 0$. 
\end{enumerate}

Then we arrive easily to the following proposition:

\begin{proposition}\label{prop-2.3}
For a toy-model operator \textup{(\ref{eqn-2.1})} with $d\ge 1$   the following estimate holds
\begin{gather}
|e^\T _{\epsilon_1,h}(x,y,\tau) - e^\T _{T,h}(x,y,\tau) |\le Ch^{\frac{1-2d}{3}}
\label{eqn-2.12}
\end{gather}
with $T=\max(\epsilon _2\sqrt{\ell(x,y)}, \, h^{\frac{1}{3}-\delta})$. 

In particular, the right-hand expression is $O(h^{1-d})$ for $d\ge 3$,  $O(h^{-\frac{4}{3}})$ for $d=2$ and  $O(h^{-\frac{2}{3}})$ for $d=1$.
\end{proposition}

It shows that case $d=2$ requires special consideration. In this case
\begin{multline*}
\partial_\tau J(\tau, x_1,y_1, z_2)\\
= \frac{i}{4\pi^2} \iiint (\xi_1-\eta_1) e^{i  (x_1\xi_1- y_1\eta_1 +z_2\xi_2  +(\tau -\xi_2^2/2) (\xi_1-\eta_1)  -(\xi_1^3-\eta_1^3)/6)}\,d\xi_1d\eta_1d\xi_2\
\end{multline*}
which after substitution $\xi_1=\alpha+\beta$, $\eta_1=\alpha-\beta$  becomes 
\begin{gather*}
\frac{i}{\pi^2} \iiint \beta e^{i \beta ((x_1+y_1 +2\tau )  -\xi_2^2 - \alpha^2 -\beta^2/3)+i z_2\xi_2  +i(x_1-y_1)\alpha    }\,d\alpha d\beta d\xi_2
\end{gather*}
and  calculating integrals with respect $\alpha$ and $\xi_2$ we arrive to
\begin{gather}
\partial_\tau J(\tau, x_1,y_1, z_2) 
=\frac{1}{\pi}   \int_{-\infty}^\infty e^{i  (   \beta (x_1+y_1+ 2\tau ) +\frac{1}{4\beta } ((x_1-y_1)^2+z_2^2) - \frac{1}{3} \beta^3)}\,d\beta .
\label{eqn-2.13}
\end{gather}

One can see easily that 

\begin{claim}\label{eqn-2.14}
Stationary points
\begin{gather*}
\beta^2=\frac{1}{2}\Bigl( x_1+y_1 +2\tau)\pm \sqrt{4x_1y_1-z_2^2}\Bigr)
\end{gather*}
are  non-degenerate as $4(x_1+\tau) (y_1+\tau) >z_2^2$ and there are no  real stationary points as 
$4(x_1+\tau) (y_1+\tau )< z_2^2$ .  
\end{claim}

\begin{proposition}\label{prop-2.4}
Consider toy-model operator \textup{(\ref{eqn-2.1})} with $d= 2$. 
Let $x_1\ge 0$, $y_1\ge 0$,  $\ell(x,y)\coloneqq \max(|x-y|,\, x_1, \,y_1 )\ge C_0h^{\frac{2}{3}}$. 
\begin{enumerate}[wide,label=(\roman*),  labelindent=0pt]
\item
If  $\ell^0(x,y)\coloneqq |x-y|\le \epsilon ' \ell(x,y)$ with $\epsilon' =\epsilon'(\epsilon)>0$ then the following estimate holds 
\begin{gather}
|e^\T _{\epsilon_1,h}(x,y,0) - e^\T _{T,h}(x,y,0)| \le Ch^{-\frac{1}{2}}\ell(x,y)^{-\frac{5}{4}}  
\label{eqn-2.15}
\end{gather}
with $T=C_0\sqrt{\ell^0(x,y)}$.
\item
Further, if $\ell^0(x,y)\ge \epsilon \ell(x,y)$ then the following estimates holds:
\begin{gather}
|e^\T _{\epsilon_1,h}(x,y,0)| \le Ch^{-\frac{2}{3}}\ell(x,y)^{-1}.
\label{eqn-2.16}
\end{gather}
\item
More precisely, if $\ell^0(x,y)\ge \epsilon \ell(x,y)$ and 
\begin{gather}
(x_2-y_2)^2 \le 4x_1y_1 -\epsilon \ell (x,y)^2
\label{eqn-2.17}\\
\shortintertext{then}
|e^\T _{\epsilon_1,h}(x,y,0)| \le Ch^{-\frac{1}{2}}\ell(x,y)^{-\frac{5}{4}}.
\label{eqn-2.18}
\end{gather}
\item
On the other hand, if
\begin{gather}
(x_2-y_2)^2 \ge 4x_1y_1 +\epsilon \ell (x,y)^2
\label{eqn-2.19}
\end{gather}
 then this estimate and all similar estimates below acquire  factor $(\ell (x,y) h^{-\frac{2}{3}})^{-s}$:
\begin{gather}
|e^\T _{\epsilon_1,h}(x,y,0)| \le Ch^{(2s-4)/3}\ell(x,y)^{-s}.
\label{eqn-2.20}
\end{gather}
\end{enumerate}
\end{proposition}

\begin{proof}
Integration by $\tau$ and multiplication by $h^{-1}$ results in the formal expression 
\begin{gather}
\frac{1} {8\pi^2h^2i} \iiint (\xi_1-\eta_1)^{-1} e^{ih^{-1}\Phi (x_1,y_1,z'; \xi_1,\eta_1,\xi';0) }\,d\xi_1d\eta_1d\xi'
\label{eqn-2.21}
\end{gather}
and we need to apply the stationary phase principle counting only points (in $\bC^3$) corresponding to $|\xi_1-\eta_1|\gtrsim \sqrt{\ell }$. Then 
\begin{enumerate}[wide,label=(\alph*),  labelindent=0pt]
\item
Due to above analysis in the \emph{regular zone} described in Statements~(i) and (iii) all three eigenvalues of  $\Hess_{\xi_1,\eta_1,\xi_2} \Phi$ are
real and $\asymp \sqrt{\ell}$ and therefore the left-hand expressions in   (\ref{eqn-2.15}) and (\ref{eqn-2.18}) do  not exceed 
$Ch^{-2} \ell^{-\frac{1}{2}} (h \ell^{-\frac{1}{2}})^{\frac{3}{2}}$.
 
\item
On the other hand, in the \emph{singular zone} where 
\begin{gather}
|4x_1y_1 -(x_2-y_2)^2|\le \epsilon \ell(x,y)^2
\label{eqn-2.22}
\end{gather}
 this is true only for two eigenvalues and therefore the left-hand expression in (\ref{eqn-2.16}) does not exceed $Ch^{-2} \ell^{-\frac{1}{2}} (h \ell^{-\frac{1}{2}})^{\frac{2}{2}}\times h^{\frac{1}{3}}$.

\item
In the \emph{shadow zone} described in Statement (iv) again all eigenvalues of Hessian are    $\asymp \sqrt{\ell}$ but only two are real and the third has imaginary part 
 $\asymp \sqrt{\ell}$  and therefore the left-hand expression in (\ref{eqn-2.18}) does not exceed $Ch^{-2}  (h \ell^{-\frac{1}{2}})^{s}$.
\end{enumerate}
\end{proof}

\begin{remark}\label{remark-2.5}
\begin{enumerate}[wide,label=(\roman*),  labelindent=0pt]
\item
We can be more precise in the analysis in the singular zone.

\item
Statements (ii)--(iv) remain true as $|x_1|\le h^{\frac{2}{3}}$ or  $|y_1|\le h^{\frac{2}{3}}$.

\item
Further, as \underline{either} $x_1< -h^{\frac{2}{3}}$ \underline{or} $y_1< -h^{\frac{2}{3}}$ \underline{or} both $x_1< -h^{\frac{2}{3}}$ and $y_1< -h^{\frac{2}{3}}$ then stationary points move to imaginary axis and estimates in Statement~(ii) and (iii)  and all similar estimates below acquire correspondingly \underline{either}  factor $ h^{\frac{2s}{3}}|x_1|^s$ \underline{or} $ h^{\frac{2s}{3}}|y_1|^s$ \underline{or} both.

\item
The right-hand expression of (\ref{eqn-2.15})  is $O(h^{-1})$ as $\ell (x,y)\ge h^{\frac{2}{5}}$ and the right-hand expression of (\ref{eqn-2.16})  is $O(h^{-1})$ as $\ell (x,y)\ge h^{\frac{1}{3}}$.

\item
Improvements similar to Proposition~\ref{prop-2.4}(i)--(iv) are possible for $d\ge 3$ as well.
\end{enumerate}
\end{remark}

Now we want to replace $e^\T_{T,h}(x,y,0)$ in the left-hand expression of (\ref{eqn-2.15}) by  $e^\W_h(x,y,0)$ described by (\ref{eqn-1.8}).

\begin{proposition}\label{prop-2.6}
\begin{enumerate}[wide,label=(\roman*),  labelindent=0pt]
\item
For a toy-model operator \textup{(\ref{eqn-2.1})} with $d\ge 3$ and \\ $\ell(x,y)\le 1$
\begin{gather}
e^\T _{\epsilon_1,h}(x,y,0)=e_h^\W (x,y,0)  +O(h^{1-d}).
\label{eqn-2.23}
\end{gather}

\item
For a toy-model operator \textup{(\ref{eqn-2.1})} with $d=2$
\begin{gather}
e^\T _{\epsilon_1,h}(x,y,0) = e^\W_h (x,y,0) + e_{\corr,h} (x,y,0) + O(h^{-1})
\label{eqn-2.24}
\end{gather}
with the correction term
\begin{multline}
e_{\corr,h} (x,y,0) \\
= \frac{1}{h^2}\int_{-\infty}^0 J\bigl(h^{-\frac{2}{3}}\tau   , h^{-\frac{2}{3}}  x_1,   h^{-\frac{2}{3}} y_1, h^{-\frac{2}{3}} (x_2-y _2)\bigr)\,d\tau -
e_h^\W (x,y,0).
\label{eqn-2.25}
\end{multline}
\end{enumerate}
\end{proposition}

\begin{proof}
\begin{enumerate}[wide,label=(\alph*),  labelindent=0pt]
\item
As $\ell (x,y)\le h^{\frac{2}{3}}$ we have  both $e^\T _{\epsilon_1,h}(x,y,0)=O(h^{-\frac{2d}{3}})$ and\\ $e_h^\W (x,y,0)=O(h^{-\frac{2d}{3}})$.
Those are $O(h^{-\frac{4}{3}})$ for $d=2$ and $O(h^{1-d})$ for $d\ge 3$.
\item
As $\ell (x,y)\ge h^{\frac{2}{3}}$ due to rescaling $x\mapsto \ell^{-1}x$, $y\mapsto \ell^{-1}y$, $\tau\mapsto \ell^{-1}\tau$ and $h\mapsto \hbar = \ell^{-\frac{3}{2}}h$ we have 
\begin{gather}
e^\T _{\epsilon_1,h}(x,y,0)-e_h^\W (x,y,0)=O\bigl( \hbar^{1-d}  \times \ell^{-d}\bigr)= O\bigl( h^{1-d}\ell^{\frac{d-3}{2}}\bigr)
\label{eqn-2.26}
\end{gather}
which is $O(h^{-\frac{4}{3}})$ for $d=2$, $O(h^{-2})$ for $d=3$ and $O(h^{1-d})$ for $d\ge 4$, $\ell\le 1$. Statement (i) has been proven.
\item
As $d=2$ for a toy-model operator   we can replace with $O(h^s)$ error $e^\T_{\epsilon_1,h}(x,y,0)$ by $e^\T_{\infty,h}(x,y,0)$ which due to (\ref{eqn-2.7}) is exactly the first term  in the right-hand expression in (\ref{eqn-2.25}); so Statement (ii) is just a definition.
\end{enumerate}
\end{proof}

\begin{remark}\label{remark-2.7}
\begin{enumerate}[wide,label=(\roman*),  labelindent=0pt]
\item
Due to (\ref{eqn-2.13}) the first term in the right-hand expression of (\ref{eqn-2.25}) can be rewritten as
\begin{gather}
-\frac{2}{\pi}\int _{-\infty}^0 \int_{0}^\infty \tau \cos\Bigl(\beta (x_1+y_1+2\tau) +\frac{1}{4\beta}  \bigl((x_1-y_1)^2 +z_2^2\bigr)-\frac{1}{3}\beta^3\bigr) \,d\beta d\tau.
\label{eqn-2.27}
\end{gather}
As $x=y$ it becomes
\begin{multline}
-\frac{2}{\pi}\int _{-\infty}^0 \int_{0}^\infty \tau \cos\bigl(2\beta (x_1+ \tau) +\frac{1}{3}\beta^3\bigr) \,d\beta d\tau\\
= -2  \int_{-\infty}^0 \tau \Ai (-2(x_1+\tau))\,d\tau
\label{eqn-2.28}
\end{multline}
with Airy function $\Ai(\cdot)$.

\item
One can prove easily that 
\begin{multline}
e^\W(x,y,\tau) = \frac{1}{4\pi^2r^2}\int_0^{2\pi}\int_0^{h^{-1}r\rho}  \cos (\sigma\cos(\theta))\,\sigma d\sigma d\theta\\
\text{with\ \ } \rho =\sqrt{2(\tau V((x+y)/2))}, \quad r=|x-y|
\label{eqn-2.29}
\end{multline}
and  in the framework of Proposition~\ref{prop-2.4}(ii) 
\begin{gather}
|e^\W_h (x,y,0)|\le Ch^{-\frac{1}{2}}\ell (x,y)^{-\frac{5}{4}}.
\label{eqn-2.30}
\end{gather}

\item
In Proposition~\ref{prop-3.6} in much more general settings will be proven that in the framework of Proposition~\ref{prop-2.4}(i) 
$e^\W(x,y,0)$ provides a good approximation for $e^\T_{T,h}(x,y,0)$ with $T=\epsilon '\sqrt{\ell(x,y)}$.
\end{enumerate}
\end{remark}

\section{Generalized toy-model}
\label{sect-2.3}

Let us consider now a generalized toy-model operator with $d= 2$, $g^{jk}=\updelta_{jk}$, $V_2(x)=0$, $V_1= - kx_1$  and $V(x)=-x_1-\frac{1}{2}k^2 x_1^2$:
\begin{multline}
\bar{A}_{k,h} \coloneqq \frac{1}{2}h^2D_1^2 +   \frac{1}{2}h^2(D_2-kx_1)^2  +V(x) \\
= \frac{1}{2}h^2D_1^2 +\frac{1}{2}h^2D_2 -(1+khD_2) x_1
\label{eqn-2.31}
\end{multline}
As in Subsection~\ref{sect-2.2}   considering equation to propagator and making $h$-Fourier transform we get equation similar to (\ref{eqn-2.4}):
\begin{gather}
\Bigl(\tau - \frac{1}{2} \xi_1^2 - \frac{1}{2}\xi_2^2 + ih (1+k\xi_2)\partial _{\xi_1}\Bigr) \widehat{u_h^\pm} = \mp i  e^{ih^{-1}(-y_1\xi_1- y_2,\xi_2)}  
\label{eqn-2.32}
\end{gather}
as  $\mp \Im \tau >0$; rewriting it as
\begin{multline*}
\partial_{\xi_1} \Bigl(e^{-ih^{-1}(1+ k\xi_2)^{-1}((\tau-\xi_2^2 /2) \xi_1-\xi_1^3/6)} \widehat{u_h^\pm} \Bigr)=\\
 \mp (1+ k\xi_2)^{-1}e^{-ih^{-1}(1+ k\xi_2)^{-1}((\tau-\xi_2^2/2) \xi_1 -\xi_1^3/6)-ih^{-1}(y_1\xi_1+y_2\xi_2) },
\end{multline*}
we arrive to expression similar to (\ref{eqn-2.5})
\begin{multline}
 \widehat{u_h^\pm} (\xi,y,\tau)
 = \mp  (1+k\xi_2)^{-1} e^{ih^{-1}((\tau-x_2^2/2) \xi_1-\xi_1^3/6)}\\
 \times 
 \int _{\mp \infty}^{\xi_1} e^{-ih^{-1}(1+ k\xi_2)^{-1}((\tau-\xi_2^2/2)\eta_1-\eta_1^3/6) -ih^{-1} (y_1\eta_1+y_2\xi_2) }\,d\eta_1
 \label{eqn-2.33}
 \end{multline}
 and making partial inverse  $h$-Fourier transform we arrive to the final answer similar to (\ref{eqn-2.6}):
 \begin{multline}
 F_{t\to h^{-1}\tau} u_h^\pm 
 = \mp (2\pi h)^{-2} \iint (1+ k\xi_2)^{-1} \\
 \times
 \int _{\mp \infty}^{\xi_1} e^{ih^{-1}(1+ k\xi_2)^{-1}((\tau-\xi_2^2/2) (\xi_1 -\eta_1) -(\xi_1^3-\eta_1^3)/6)+ ih^{-1}(x_1\xi_1-y_1\eta_1+(x_2-y_2)\xi_2) }
 \, d\eta_1d\xi_1d\xi_2.
 \label{eqn-2.34}
\end{multline}
Then as $\tau \in \bR$ we get (cf. (\ref{eqn-2.7}))
\begin{multline}
 F_{t\to h^{-1}\tau} u_h  
 = (2\pi h)^{-2} \iiint (1+ k\xi_2)^{-1} \\
 e^{ih^{-1}(1+ k\xi_2)^{-1}((\tau-\xi_2^2/2) (\xi_1 -\eta_1) -(\xi_1^3-\eta_1^3)/6)+ ih^{-1}(x_1\xi_1-y_1\eta_1+(x_2-y_2)\xi_2) } \, d\eta_1d\xi_1d\xi_2.
  \label{eqn-2.35}
\end{multline}

We can see easily that for $|k| \le C$ and $\min(|x_1|,\,|y_1|)\le \epsilon$  the phase function 
\begin{multline}
\Phi _k (x_1,y_1,z_2; \xi_1,\eta_1,\xi_2;\tau) \\
= x_1\xi_1-y_1\eta_1+z_2\xi_2  +(1+ k \xi_2)^{-1} [( \tau -\frac{1}{2}\xi_2^2 ) (\xi_1-\eta_1) -\frac{1}{6}(\xi_1^3-\eta_1^3)]
\label{eqn-2.36}
\end{multline}
has four stationary points (with respect to $(\xi_1,\eta_1,\xi_2)\in \bC^3$): on all of them $|\xi_2|\le C_0\sqrt{\epsilon}$, 
\begin{gather}
2x_1=(1+ k \xi_2)^{-1}[-2\tau +\xi_1^2 + \xi_2^2],
\label{eqn-2.37}\\
2y_1=(1+ k \xi_2)^{-1}[-2\tau +\eta_1^2+\xi_2^2],
\label{eqn-2.38}
\end{gather}
and 
\begin{multline}
z_2 = \xi_2 (1+ k \xi_2)^{-1} (\xi_1-\eta_1)\\
+ k (1+ k \xi_2)^{-2}\bigl[( \tau -\frac{1}{2}\xi_2^2 ) (\xi_1-\eta_1) -\frac{1}{6}(\xi_1^3-\eta_1^3)\bigr]
\label{eqn-2.39}
\end{multline}
and for corresponding trajectories  on the energy level $\tau$
\begin{gather}
T(x,y)= (1+ k \xi_2)^{-1} |\xi_1-\eta_1|.
\label{eqn-2.40}
\end{gather}
Then the same conclusion as for $k =0$ holds:

\begin{enumerate}[wide,label=(\alph*),  labelindent=0pt]
\item
Proposition~\ref{prop-2.3} remains true;
\item
All statements of Proposition~\ref{prop-2.4} remain true with
 regular, singular and shadow zones defined exactly by (\ref{eqn-2.17}), (\ref{eqn-2.22}) and
(\ref{eqn-2.19}).
\item
All statements of Remark~\ref{remark-2.5} remain true.
\end{enumerate}
Again consider $\partial_\tau J_k(\tau, x_1, y_1,z_2)$ this time based on (\ref{eqn-2.35}):
\begin{multline*}
\frac{i}{4\pi^2  i}
 \iiint   (1+ \kappa \xi_2)^{-2}(\xi_1-\eta_1) \\
 \times e^{i (1+ \kappa \xi_2)^{-1}( \tau-\xi_2^2 /2) (\xi_1 -\eta_1) -(\xi_1^3-\eta_1^3)/6)+ ih^{-1}(x_1\xi_1-y_1\eta_1+z_2\xi_2) } \, d\eta_1d\xi_1d\xi_2.
 \end{multline*}
  Plugging $\xi_1=\alpha+\beta$, $\eta_1=\alpha-\beta$ and then substituting $\beta\coloneqq (1+\kappa \xi_2)\beta$ we arrive to 
   \begin{multline}
\frac{i}{\pi^2  }
 \iiint   d\alpha d\beta d\xi_2\times \\[3pt]
 \beta \exp \Bigl[ i \Bigl( (2\tau-\xi_2^2  -\alpha^2 ) \beta  -\frac{1}{3} (1+\kappa \xi_2)\beta^3  +(1+\kappa \xi_2) (x_1+y_1)\beta +z_2\xi_2 +(x_1-y_1)\alpha \Bigr)\Bigr] .
 \label{eqn-2.41}
 \end{multline}
Then we can calculate integrals with respect to $\alpha$ and $\xi_2$, arriving to formula similar to (\ref{eqn-2.27}):
\begin{multline}
-\frac{1}{\pi }
 \int_{-\infty}^\infty   d\beta \times \\[3pt]
\exp \Bigl[i \Bigl(    (2\tau+x_1+y_1)\beta -\frac{1}{3}\beta^3  +\frac{1}{4\beta}   (x_1-y_1)^2 + \frac{1}{4\beta} \bigl(z_2 - \underbracket{\frac{1}{3}\kappa \beta^3 -\kappa (x_1+y_1)\beta }   \bigr)^2      \Bigr)\Bigr]   .
 \label{eqn-2.42}
 \end{multline} 
 \begin{enumerate}[wide,label=(\alph*),  labelindent=0pt]
\item
 Consider first  $\beta \asymp \sqrt{\ell}$. Observe that selected terms in the phase in (\ref{eqn-2.43}) are $O(\ell^{\frac{3}{2}})$ and squared and divided by $\beta$  they are $O(\ell^\frac{5}{2})$ and therefore assuming that $\ell\le h^{\frac{2}{5}}$  we arrive to the following conclusion:   if we neglect their square the error will not exceed  
\begin{gather}
\left\{\begin{aligned}
&Ch^{-\frac {1}{2}} \ell^{-\frac{5}{4}} \times \ell^{\frac{5}{2}}h^{-1}=O(h^{-1})               &&\text{ in the regular zone},\\
&Ch^{-\frac {2}{3}} \ell^{-1} \times \ell^{\frac{5}{2}}h^{-1}=O(h^{-\frac{5}{3}}\ell^{\frac{3}{2}}) &&\text{ in the singular zone},\\
&Ch^{(2s-2)/3}\ell ^{-s} &&\text{ in the shadow zone}.
\end{aligned}\right.
\label{eqn-2.43}
\end{gather}
\item
On the other hand, consider $\beta \epsilon \sqrt{\ell}$ and therefore $\ell^0\coloneqq |z_2|+|x_1-y_1|\le \ell$, and then one can estimate  the error in the cases
$\ell^0\sqrt{\ell}\ge h$ and  $\ell^0\sqrt{\ell}\le h$ and in both cases the error will be smaller than given by the first line of (\ref{eqn-2.43}).
\end{enumerate}

 But then (\ref{eqn-2.42}) becomes
  \begin{multline}
 \partial_\tau J_\kappa (\tau,x_1,y_1,z_2)\equiv -\frac{2}{\pi }  \int_0^\infty   
\cos \Bigl( (2\tau+x_1+y_1)\beta -\frac{1}{3}\beta^3  +\frac{1}{4\beta}   [(x_1-y_1)^2 +z_2^2]\Bigr)\\
\times \exp\Bigl[ i\Bigl(- \frac{1}{6} \kappa  z_2 \beta^2 -\frac{1}{2} \kappa  z_2(x_1+y_1)  \Bigr)\Bigr]\,d\beta  
 \label{eqn-2.44}
 \end{multline} 
 and and scaling back we arrive  to
 
 \begin{proposition}\label{prop-2.8}
Consider generalized toy-model operator \textup{(\ref{eqn-2.31})} in dimension $d=2$. Then as $h^{\frac{2}{3}}\le \ell \le h^{\frac{2}{5}}$ with an error \textup{(\ref{eqn-2.42})}
\begin{gather}
e^\T _{\epsilon_1,h}(x,y,0)\equiv e^\W_h(x,y,0) + e_{\corr, k,h} (x,y,0)
 \label{eqn-2.45}
 \end{gather}
with the correction term
\begin{multline}
e_{\corr,k, h} (x,y,0) \\
= \frac{1}{h^2}\int_{-\infty}^0 J_{h^{\frac{1}{3}}k} \bigl(h^{-\frac{2}{3}}\tau   , h^{-\frac{2}{3}}  x_1,   h^{-\frac{2}{3}} y_1, h^{-\frac{2}{3}} (x_2-y _2)\bigr)\,d\tau -
e_h^\W (x,y,0) 
\label{eqn-2.46}
\end{multline}
where  $ \partial_\tau J_\kappa (\tau,x_1,y_1,z_2)$ defined by \textup{(\ref{eqn-2.44})} and $e_h^\W (x,y,0) $ by \textup{(\ref{eqn-2.29})}.
\end{proposition}

\begin{remark}\label{remark-2.9}
Obviously 
$ J_\kappa (\tau,x_1,y_1,0)=J_0(\tau,x_1,y_1,0)$.
\end{remark}

\chapter{General case}
\label{sect-3}

In the framework of Theorems~\ref{thm-1.1} and~\ref{thm-1.2} without any loss of the generality one can assume that 
\begin{phantomequation}\label{eqn-3.1}\end{phantomequation}
\begin{phantomequation}\label{eqn-3.2}\end{phantomequation}
\begin{phantomequation}\label{eqn-3.3}\end{phantomequation}
\begin{align}
&V(x) =- x_1 W(x),  &&W(0)=1,
\tag*{$\textup{(\ref*{eqn-3.1})}_{1,2}$} \\
&\bar{x}=(\bar{x}_1, 0)    &&\bar{y}=(\bar{y}_1,\bar{y}'), &&\bar{y}_1\le  \bar{x}_1,
\tag*{$\textup{(\ref*{eqn-3.2})}_{1-3}$}\\
&g^{j1}=\updelta_{j1}, &&g^{jk}(0) =\updelta_{jk}, && V_1=0, &&V_j(0)=0,
\tag*{$\textup{(\ref*{eqn-3.3})}_{1-4}$}
\end{align}
where $\bar{x}$ and $\bar{y}$ are two ``target'' points. Indeed,  we can reach this by a change of variables. We will assume that these assumptions are fulfilled until the end of this section. Then according to (\ref{eqn-1.12}) 
\begin{gather}
\ell(x,y)\asymp \max (|x_1|,\,|y_1|,\, |x'-y'|,\, h^{\frac{2}{3}}).
\label{eqn-3.4}
\end{gather}

\section{Preliminary remarks}
\label{sect-3.1}

\begin{proposition}\label{prop-3.1} 
\begin{enumerate}[wide,label=(\roman*),  labelindent=0pt]
\item
Estimate holds
\begin{gather}
|F_{t\to h^{-1}\tau} \Bigl(\chi_T(t) u_h(x,y,t)\Bigr)|\le Ch^{1-d} \ell^{(d-2)/3}\hbar^{s}
\label{eqn-3.5}
\end{gather}
with $|\tau |\le  \epsilon T^2$, $T=\max(C_0\ell ^{\frac{1}{2}},\, h^{\frac{1}{3}})$ and $\hbar = T^{-3}h$.
\item
Furthermore, if $\ell^0 =\ell^0(x,y)\coloneqq |x-y| \ge \epsilon_0 \ell  $ and $\ell \ge h^{\frac{2}{3}}$, then in the same framework
\begin{gather}
|F_{t\to h^{-1}\tau} \Bigl(\bar{\chi}_T(t) u_h(x,y,t)\Bigr)|\le Ch^{1-d} \ell^{(d-2)/3} \hbar^{s}.
\label{eqn-3.6}
\end{gather}
\end{enumerate}

Here and below $\ell =\ell(x,y)$, 
$\chi \in \sC_0^\infty ([-1,-\frac{1}{2}]\cup [\frac{1}{2},1])$, $\bar{\chi }\in \sC_0^\infty ([-1,1])$ equal $1$ on $[-\frac{1}{2},\frac{1}{2}]$, $s$ is an arbitrarily large exponent and $\delta>0$ is an arbitrarily small exponent.
\end{proposition}

\begin{proof}
\begin{enumerate}[wide,label=(\alph*),  labelindent=0pt]
\item
Consider first $T$ which is a small constant: $T=\epsilon_3$  (which we can assume is $1$; otherwise we achieve it by scaling) and assume only that $T\ge C_0\sqrt{\ell}$.
Then Statement (i) follows from the fact that the propagation speed with respect to $\xi_1$ is disjoint from $0$ plus ellipticity arguments. 

Statement (ii) follows from the fact that he propagation speeds with respect to $x$ and $\xi $ are bounded.

\item
Then we scale $x\mapsto T^{-2}x$, $y\mapsto T^{-2}y$, $\ell \mapsto T^{-2}\ell$, $\tau\mapsto T^{-2}\tau$ and  $h\mapsto \hbar = T^{-3}h$.
\end{enumerate}
\end{proof}

\begin{proposition}\label{prop-3.2}
The following estimates hold for $\tau\colon |\tau|\le \epsilon$, $T\colon h^{\frac{2}{3}}\le T\le \epsilon_1$:
\begin{gather}
|F_{t \to h^{-1}\tau }\Bigl(\bar{\chi}_T(t)u_h(x,y,t)\Bigr)|\le Ch^{1-d},
\label{eqn-3.7}\\
|e_h(x,y,\tau+h)-e_h(x,y,\tau)\le Ch^{1-d},
\label{eqn-3.8}\\
\intertext{and for $T\ge  T^*\coloneqq C_0\max(\sqrt{\ell},\, h^{\frac{1}{3}})$}
|e_h(x,y,0)-e^\T_{T,h} (x,y,0)|\le Ch^{-2d/3}\hbar^s+Ch^{1-d} 
\label{eqn-3.9}
\end{gather}
again with $\hbar= T^{-3}h$.
\end{proposition}

\begin{proof}
\begin{enumerate}[wide,label=(\alph*),  labelindent=0pt]
\item
It is sufficiently to prove for $\tau=0$ since shift of $V$ by $\tau\colon |\tau |\le \epsilon_2$ preserves conditions\footnote{\label{foot-3} Albeit changes $x_1$ and $\ell(x,y)$.}.
Simple scaling shows that if $\ell^0(x,y)\le \epsilon \ell(x,y)$ then with $T=T^*$  the left-hand expression of (\ref{eqn-3.7}) with $T=T^*$ does not exceed  $Ch^{1-d}T^{*\, 2(d-2)/3}$. 

On the other hand, considering partition by $T\colon T^* \le T\le \epsilon_1$ we see from Proposition~\ref{prop-3.1}(i) that the same  left-hand expression but with $\bar{\chi}_T(t)$ replaced by $\bar{\chi}_T(t)-\bar{\chi}_{T^*}(t)$ also does not exceed this. So, with indicated $T$ the same left-hand expression does not exceed this, which is $O(h^{1-d})$. 

Finally, for $\ell^0(x,y)\ge \epsilon \ell(x,y)$ the same estimate (\ref{eqn-3.7}) holds again due to  (\ref{eqn-3.5}).

\item
In particular, taking $x=y$ and $T=\epsilon_1$ and applying standard Tauberian arguments, Part I, we conclude that (\ref{eqn-3.8}) holds with $x=y$. And then it holds for $x\ne y$ due to
\begin{multline*}
|e_h(x,y,\tau+h)-e_h(x,y,\tau)|\\ \le |e_h(x,x,\tau+h)-e_h(x,x,\tau)|^{\frac{1}{2}}  |e_h(y,y,\tau+h)-e_h(y,y,\tau)|^{\frac{1}{2}} .
\end{multline*}
\item
Then (\ref{eqn-3.9}) holds with $T=\epsilon_1$ due to standard Tauberian arguments, Part II. Then due to Proposition~\ref{prop-3.1}(i) 
\begin{gather*}
|e^\T_{2T,h}(x,y,0) -e^\T_{T,h}(x,y,0)|\le CT^{\prime\, d-3}\hbar^s
\end{gather*}
and summation by partition over $[T,\epsilon_1]$ results in the same answer  which implies (\ref{eqn-3.9}).
\end{enumerate}
\end{proof}

Therefore  
\begin{claim}\label{eqn-3.10}
To evaluate $e_h(x,y,\tau)$ modulo $O(h^{1-d})$ it is enough to consider only $T\le  C_0\max(\sqrt{\ell(x,y)},\, h^{\frac{1}{3}-\delta})$. 
\end{claim}

\section{Propagator as an oscillatory integral}
\label{sect-3.2}

\begin{proposition}\label{prop-3.3}
Let  $x,y\in B(0,\epsilon  )$, $C_0\sqrt{\ell} \le T\le C_0 \epsilon )$ and $|\tau|\le \epsilon$. Then
\begin{multline}
F_{t\to h^{-1}\tau} \Bigl(\bar{\chi}_T(t)u_h(x,y,t)\Bigr) \\
=
(2\pi h)^{-d} \iint e^{i h^{-1}\varphi(x,y, \xi_1,\eta_1, \xi',\tau)} b  (x,y, \xi_1,\eta_1,\xi',\tau, h)\,d\xi_1d\eta_1 +O(h^s)
\label{eqn-3.11}
\end{multline}
where phase function $\varphi(x,y, \xi_1,\eta_1, \xi',\tau)$ differs from the phase function $\bar{\varphi} (x,y, \xi_1,\eta_1, \xi',\tau)$  in \textup{(\ref{eqn-2.7})}
by $O(T^4)$ and amplitude $b  (x,y, \xi_1,\eta_1,\xi',\tau,h)$ differs from the  amplitude $\bar{b}_{k} (x,y, \xi_1,\eta_1,\xi',\tau,h)$ there by $O(T)$.
\end{proposition}

\begin{proof}
\begin{enumerate}[wide,label=(\alph*),  labelindent=0pt]
\item
Let us start from the formal construction.  Let us  apply $h$-Fourier transform  with respect to $x_1$, $y_1$: 
$\hat{u}_h (\xi_1, x',\eta_1, y',t)=F_{x_1\to h^{-1}\xi_1, y_1\to  -h^{-1}\eta_1} u_h$.  Then the problem 
\begin{align}
&(hD_t-A(x,hD_x) )u_h (x,y,t)=0,
\label{eqn-3.12}\\[3pt]
&u_h(x,y,0)=\updelta (x-y)
\label{eqn-3.13}\\
\shortintertext{becomes}
&(hD_t-A(-hD_{\xi_1},x',\xi_1, hD_{x'}) )\hat{u}_h (\xi_1, x', \eta_1,y' ,t)=0,
\label{eqn-3.14}\\[3pt]
&\hat{u}_h (\xi_1, x', \eta_1,y' ,t)=2\pi h \updelta (\xi_1-\eta_1)\updelta (x'-y').
\label{eqn-3.15}.
\end{align}

This operator is $x_1$-microhyperbolic where $x_1$ is a variable dual to $\xi_1$.
Then using the standard arguments\footnote{\label{foot-4} See f.e. \cite{shubin:spectral}, Theorem 20.1 or  \cite{OOD}, Section \ref{OOD-sect-2}.}
 we can construct $\hat{u}_h (\xi_1, x', \eta_1,y' ,t)$
near energy level $0$ for   $x,y \in B(0,\epsilon_1)$  as a ``simple'' oscillatory integral
\begin{multline}
\hat{U}_h (\xi_1, x', \eta_1,y' ,t)\\
= (2\pi h)^{1-d}\int  e^{ih^{-1} (\psi  (\xi_1, x', \eta_1,y' , \theta , t) + a(-\theta_1, y' ,\theta')  t)} b ( \xi_1, x', \eta_1,y' , \theta  ,h ) \,d\theta ,
\label{eqn-3.16}
\end{multline}
with the phase satisfying
\begin{align}
&a(-\partial_{\xi_1}\psi  , x' ,  \partial_{x'}\psi  )= a(-\theta_1, y'  ,\theta'),
\label{eqn-3.17}\\
&\psi  |_{\xi_1=\eta_1}= \langle x'-y',\theta'\rangle ,
\label{eqn-3.18}\\
&\nabla _{\xi_1, x'} \psi  |_{\xi_1=\eta_1, x'=y'}= \theta
\label{eqn-3.19}
\end{align}
and with amplitude $b  (\xi_1,x',\eta_1, y', \theta, h)$ decomposing into asymptotic series
\begin{gather}
b  (\xi_1,x',\eta_1, y', \theta, h)\sim \sum_{n\ge0} b_n (\xi_1,x',\eta_1, y',\theta) h^n.
\label{eqn-3.20}
\end{gather}
Then for $|\tau|\le \epsilon$
\begin{multline}
F_{t\to h^{-1}\tau} \hat{U}_h(x,y,t)\\
= (2\pi h)^{2-d}  \int_{\Sigma (\xi_1,y', \tau)} e^{ih^{-1} (\psi  (\xi_1, x', \eta_1,y' , \theta , \tau ) )} b ( \xi_1, x', \eta_1,y' ,\theta ,h ) \,d\theta :d _\theta a 
\label{eqn-3.21}
\end{multline}
where $\Sigma (\xi_1,y', \tau) =\{ \theta\colon a(-\theta_1, y' , y',\theta')=\tau\}$. Due to  $x_1$-microhyperbolocity  we can express $\theta_1$ as function of 
$\xi_1, \eta_1, x',y', \tau$ and rewrite the right-hand expression as an integral over $\theta' \in \bR^{d-1}$.

Finally, making inverse Fourier transform by $\xi_1,\eta_1$ and plugging $\theta'=\xi'$ we arrive to (\ref{eqn-3.11}) for $u_h$  with
\begin{gather*}
\varphi (x, y, \xi_1,\eta_1,\xi',\tau) = \psi (\xi_1,x',\eta_1,y', \theta_1(\xi_1,x',\eta_1,y',\xi' \tau),\xi', \tau) + x_1\xi_1-y_1\eta_1.
\end{gather*}

\item
To justify (\ref{eqn-3.11})  for $u_h$ we observe that for given $x,y\in B(0,\epsilon  )$ and $T\le C_0\epsilon$ behaviour of$a(x,\xi)$ for $|x|\ge c\epsilon$ does not matter.

\item
Finally, since $\varphi =O(T^3)$ and relative error of $A_h$ in comparison to generalized toy-model operator $\bar{A}_{h}$ is $O(T)$ we conclude that the relative errors in 
$\varphi$ and $B_h$ are also $O(T)$ and therefore absolute errors are $O(T^4)$ and $O(T^2)$ correspondingly.
\end{enumerate}
\end{proof}

Therefore we arrive immediately to the following statements:

\begin{proposition}\label{prop-3.4}
Let operator  $A_h$ satisfy $\textup{(\ref{eqn-3.1})}_{1,2}$, $\textup{(\ref{eqn-3.2})}_{1-3}$ and $\textup{(\ref{eqn-3.3})}_{1-4}$. Then
\begin{enumerate}[wide,label=(\roman*),  labelindent=0pt]
\item
In the framework of Proposition~\ref{prop-2.3} estimate \textup{(\ref{eqn-2.12})} holds.

\item
In the framework of Proposition~\ref{prop-2.4}(i) estimate \textup{(\ref{eqn-2.15})} holds.

\item
In the framework of Proposition~\ref{prop-2.4}(ii) estimate  \textup{(\ref{eqn-2.16})} holds.

\item
In the framework of Proposition~\ref{prop-2.4}(iii) estimate    \textup{(\ref{eqn-2.18})} holds.

\item
In the framework of Proposition~\ref{prop-2.4}(iv) estimate    \textup{(\ref{eqn-2.20})} holds.
\end{enumerate}
\end{proposition}

Also observe that all statements of Remark~\ref{remark-2.5} remain valid.

\begin{corollary}\label{corollary-3.5}
Let  $d\ge 3$.  Then  $e_h(x,y,\tau)=e^\W_h(x,y,\tau) +O(h^{1-d})$. Furthermore 
$e_h(x,y,0)=O(h^{1-d})$ unless $\eff(x)\asymp \eff (y)\ge h^{\frac{2}{3}}$ and $\ell^0(x,y)\le \epsilon \eff(x)$.
\end{corollary}

\begin{proof}
If $\ell ^0(x,y)\le \epsilon  \ell(x,y)$, $\ell (x,y)\ge h^{\frac{2}{3}}$ we prove estimate~(\ref{eqn-1.7}) by simple rescaling. Otherwise the simple rescaling shows that 
$|e^\T_{T,h} (x,y,0)|\le Ch^{1-d}$ and therefore $|e_h (x,y,0)|\le Ch^{1-d}$ and  one can see easily that in this case  also $|e^\W _h(x,y,0)|\le Ch^{1-d}$  due to stationary phase principle. 
\end{proof}

So, Theorem~\ref{thm-1.1} has been proven. From now we consider only $d=2$.

\begin{proposition}\label{prop-3.6}
Let $d=2$,  $\ell (x,y)  \le \epsilon_1$ and $\ell^0(x,y)\le \epsilon \ell(x,y)$.  Then
\begin{gather}
|e^\T_{T_*,h} (x,y,0)- e^\W _h(x,y,0)|\le Ch^{-1}+ C\ell^{-2}.
\label{eqn-3.22}
\end{gather}
\end{proposition}

\begin{proof}
Following arguments of \cite{OOD}, Section~\ref{OOD-sect-2} one can prove easily that after rescaling the left-hand expression does not exceed
\begin{gather}
C \hbar^{-\frac{1}{2}}\ell^{0\,-\frac{3}{2}}  \min \Bigl( 1 , \frac{\ell^{0\,3} \ell }{\hbar}\Bigr) + C\hbar^{-1}\ell +C 
\label{eqn-3.23}
\end{gather}
where the first term is the difference between 
\begin{gather*}
(2\pi \hbar)^{-2}   \iint e^{i\hbar^{-1} (it\tau- \varphi^0(x,y,\theta))} B_0(x,y,\theta,t)\chi_{T_*/T} (t)\,d\theta dt
\end{gather*}
and $e^\W _h(x,y,0)$. Further,  the second term comes from the amplitude  $B_1\hbar $ in the same expression  (and acquires factor $\ell$ due to scaling of  $V_j$) and the third term comes from the amplitude $B_2\hbar^2$ in the same expression,
which are in $e^\T_{T_*,h}(x,y,0)$ but are skipped in $e^\W _h(x,y,0)$\, \footnote{\label{foot-5} These amplitudes $B_n(x,y,\theta,t)$  should not be confused with the amplitudes $b_n$  in the proof of Proposition~\ref{prop-3.3}.}  . 

The first term in (\ref{eqn-3.23})  does not exceed $C\hbar^{-1} \ell ^{\frac{1}{2}}$, so the left-hand expression of (\ref{eqn-3.22}) does not exceed
$C\hbar^{-1}\ell ^{\frac{1}{2}}+ C$. Scaling back (that is multiplying  by $\ell^{-2}$ and plugging $\hbar =\ell^{-\frac{3}{2}}h$) we get $Ch^{-1}  + C\ell^{-2}$.
\end{proof}

Combining Propositions~\ref{prop-3.6} and~\ref{prop-3.4}(ii) we arrive to
 
\begin{corollary}\label{corollary-3.7}
In the framework of Proposition~\ref{prop-3.6} 
\begin{gather}
|e_{h} (x,y,0)- e^\W _h(x,y,0)|\le Ch^{-\frac{1}{2}}\ell^{-\frac{5}{4}}+C\ell^{-2} + Ch^{-1}.
\label{eqn-3.24}
\end{gather}
In particular, for $\ell \ge h^{\frac{2}{5}}$ the right-hand expression is $Ch^{-1}$.
\end{corollary}

\section{Perturbation methods}
\label{sect-3.3}

Now we compare $e^\T_{T,h} (x,y, 0)$ with the same function for generalized  toy-model operator (\ref{eqn-2.31}) assuming that 
\begin{phantomequation}\label{eqn-3.25}\end{phantomequation}
\begin{gather}
V_{x_1}(0) =-1,\qquad  V_1=0, \qquad V_{2\, x_1}(0)=k
\tag*{$\textup{(\ref*{eqn-3.25})}_{1-3}$}
\end{gather}
because we can achieve $\textup{(\ref*{eqn-3.25})}_{1}$ by scaling and $\textup{(\ref*{eqn-3.25})}_{3}$ as a definition. Then

\begin{claim}\label{eqn-3.26}
The phase function $\varphi(x,y, \xi_1,\eta_1, \xi',\tau)$ in expression (\ref{eqn-3.11}) differs from the phase function $\bar{\varphi} (x,y, \xi_1,\eta_1, \xi',\tau)$  in \textup{(\ref{eqn-2.35})}
by $O(T^5)$ and amplitude $b  (x,y, \xi_1,\eta_1,\xi',\tau, h)$ differs from the  amplitude $\bar{b}_{k} (x,y, \xi_1,\eta_1,\xi',\tau,h)$ there by $O(T^2)$.
\end{claim}
 
 Indeed, in comparison  to the Part (c) of the proof of Proposition~\ref{prop-3.4}  now the  relative error in $A_h$ in comparison to generalized toy-model operator $\bar{A}_{k,h}$ is $O(T^2)$ we conclude that the relative errors in  $\varphi$ and $B_h$ are also $O(T^2)$.

Let us assume  that
\begin{gather}
\ell (x,y) \le h^{\frac{2}{5}}.
\label{eqn-3.27}
\end{gather}

\begin{proposition}\label{prop-3.8}
Let  \textup{(\ref{eqn-3.25})} and \textup{(\ref{eqn-3.27})} be fulfilled. Then  
\begin{enumerate}[wide,label=(\alph*),  labelindent=0pt]
\item
Asymptotics 
\begin{multline}
\bigl(e^\T_{T^*,h} (x,y,0) - e^\W_{h} (x,y,0)\bigr)- \bigl(\bar{e}^\T_{k,h}(x,y,0)- \bar{e}^\W_{k,h} (x,y,0)\bigr) \\
= O(h^{-1})
\label{eqn-3.28}
\end{multline}
holds  for all   $x,y\in B(0,\epsilon )\colon |x-y|\le \epsilon \ell(x,y)$.

\item
Estimates
\begin{multline}
|e^\T_{T^*, h} (x,y, 0)-\bar{e}^\T_{k,T^*, h}(x,y,0)|\\
\le 
\left\{\begin{aligned}
&Ch^{-\frac {1}{2}} \ell^{-\frac{5}{4}} \times \ell^{\frac{5}{2}}h^{-1}=O(h^{-1})               &&\text{ in the regular zone},\\
&Ch^{-\frac {2}{3}} \ell^{-1} \times \ell^{\frac{5}{2}}h^{-1}=O(h^{-\frac{5}{3}}\ell^{\frac{3}{2}}) &&\text{ in the singular zone},\\
&Ch^{(2s-2)/3}\ell ^{-s} &&\text{ in the shadow zone}
\end{aligned}\right.
\label{eqn-3.29}
\end{multline}
and
\begin{gather}
e^\W_{ h} (x,y,0)-\bar{e}^\W_{k, h}(x,y,0) =O(h^{-1})
\label{eqn-3.30}
\end{gather}
hold for all   $x,y\in B(0,\epsilon )\colon |x-y|\ge \epsilon \ell(x,y)$ where $\bar{e}_{k,h}(x,y,\tau )$,  $\bar{e}^\W_{k,h} (x,y,\tau)$ and $ \bar{e}^\T_{k,T ,h} (x,y,\tau)$ are defined for a generalized toy-model operator \textup{(\ref{eqn-2.31})}.
\end{enumerate}
\end{proposition}

\begin{proof}
\begin{enumerate}[wide,label=(\alph*),  labelindent=0pt]
\item
Observe that $A_h-\bar{A}_{k,h} = O(\ell^2)$. Indeed,  when we replace $V$ by $x_1$ the error is $O(\ell^2)$ due to $\textup{(\ref*{eqn-3.25})}_{1}$ and when we replace
$V_1$ by $kx_1$ the error is  $O(\ell^{\frac{5}{2}})$ due to $\textup{(\ref*{eqn-3.25})}_{2-3}$ where extra factor $\ell^{\frac{1}{2}}$ comes from factor $\xi_2$. 

Therefore as  $|x-y|\le \epsilon \ell(x,y)$ the following estimate holds
\begin{multline}
|\bigl(e^\T_{T^*,h} (x,y,0 )- e^\T_{T_*,h} (x,y,0 )\bigr) - \bigl( \bar{e}^\T_{k,T^*,h} (x,y,0)- \bar{e}^\T_{k,T_*,h} (x,y,0) \bigr)|\\
\le Ch^{-1} 
\label{eqn-3.31}
\end{multline}
and as $|x-y|\ge \epsilon \ell(x,y)$ estimate (\ref{eqn-3.30}) holds, 
where we recall that $T^*= C_0\sqrt{\ell}$ and $T_*=\epsilon ' \sqrt{\ell}$ with arbitrarily small constant $\epsilon'$. The proof is similar to the proof of Proposition~\ref{prop-2.8}
and the right-hand expression in (\ref{eqn-3.29}) is exactly expression (\ref{eqn-2.43}).

\item
Further,  if $|x-y|\le \epsilon _1 \ell(x,y)$ then combining (\ref{eqn-3.31}) and (\ref{eqn-3.22}) 
we conclude that the left-hand expression of (\ref{eqn-3.27}) does not exceed $Ch^{-1}+C\ell^{-2}$ which implies (\ref{eqn-3.27})  for $\ell \ge h^{\frac{1}{2}}$. 

However, for $\ell (x,y)\le h^{\frac{1}{2}}$ we need more subtle arguments.  Namely, recall that $C\ell^{-2}$ comes from $\hbar^0 \times \ell^{-2}$ which in turn comes from the decomposition (\ref{eqn-3.16}) with $B_n = O(\ell^{-\frac{3n}{2}})$. However since operators $A_h$ and $\bar{A}_{k,h}$ after rescaling differ by $O(\ell^2)$ rather than $O(\ell)$ we can estimate $B_n - \bar{B}_{k,n}= O(\ell^{1-\frac{3n}{2}})$ and therefore 
\begin{multline}
|\bigl(e^\T _{T_*,h} (x,y,0) - e^\W_{h} (x,y,0)\bigr)- \bigl(\bar{e}^\T_{T_*,k,h}(x,y,0)- \bar{e}^\W_{k,h} (x,y,0)\bigr)|\\
\le Ch^{-1} + C\ell^{-\frac{3}{2}}\le Ch^{-1}.
\label{eqn-3.32}
\end{multline}

\item
On the other hand,  if $\ell^0(x,y)\ge \epsilon _1 \ell(x,y)$ then     one can prove easily
  \begin{gather}
 |e^\W_{h} (x,y,0) - \bar{e}^\W_{k,h} (x,y,0)| \le Ch^{-1} + C\ell^{-\frac{3}{2}}\le Ch^{-1}.
\label{eqn-3.33}
\end{gather}
\end{enumerate}\vspace{-\baselineskip}\ 
\end{proof}

\begin{remark}\label{remark-3.9}
\begin{enumerate}[wide,label=(\roman*),  labelindent=0pt]
\item
One can prove easily that 
\begin{gather}
e^\W_{ h} (x,y,0)-\bar{e}^\W_{k,  h}(x,y,0) =O(h^{-2}\ell^2)
\label{eqn-3.34}
\end{gather}
even as $|x-y|\ge \epsilon \ell(x,y)$. Therefore
\begin{gather}
e_h(x,y,0) = \bar{e}_{k,h}(x,y,0)+O(h^{-1})\qquad \text{as\ \ } \ell(x,y)\le h^{\frac{1}{2}}.
\label{eqn-3.35}
\end{gather}
However, one can easily note that it may fail as  $\ell(x,y)\gg h^{\frac{1}{2}}$. Thus, as $h^{\frac{1}{2}}\le \ell(x,y)\le h^{\frac{2}{5}}$ we need to use both terms 
$e^\W_h(x,y,0)$ and $e_{\corr,k,h}(x,y,0)$ in the asymptotics.

\item
We can replace $k$ by $k=0$ and preserve remainder estimate  $O(h^{-1})$ if and only if $|k| |x_2-y_2|\lesssim \ell^{\frac{1}{4}}h^{\frac{1}{2}}$. In particular, we can do it as $x_2=y_2$. 

\item
We know that in any dimension $d\ge 2$ \ $e_h(x,y,0)$ is $O(h^s)$ if \underline{either} $x_1< - h^{\frac{2}{3}-\delta}$ \underline{or} $y_1< - h^{\frac{2}{3}-\delta}$ \underline{or}
$4x_1y_1\le (1-\epsilon) |x'-y'|^2 -h^{\frac{4}{3}-\delta}$ but it is not the case for $e_h^\W(x,y,0)$!
\end{enumerate}
\end{remark}

\section{Proof of Theorem~\ref{thm-1.2}}
\label{sect-3.4}

\begin{proof}[Proof of Theorem~\ref{thm-1.2}]
Without any loss of the generality one can assume that $\tau=0$, and assumptions $\textup{(\ref*{eqn-3.1})}_{1,2}$, $\textup{(\ref*{eqn-3.2})}_{1-3}$, 
$\textup{(\ref*{eqn-3.3})}_{1-4}$ and $\textup{(\ref*{eqn-3.25})}_{1-3}$ are fulfilled. Then  the left-hand expression in (\ref{eqn-1.11}) is 
$(x_1-x_2)^2-4x_1y_1 $ modulo $O(\ell^3)$.

As $\ell(x,y)\ge h^{\frac{2}{3}-\delta}$ we can replace in the Tauberian expressions $T=C_0\sqrt{\ell}$ by $T=\epsilon_1$ with
$O(h^{s})$ error. Then  
\begin{enumerate}[wide,label=(\alph*),  labelindent=0pt]
\item
Statements (i) and (ii) of Theorem~\ref{thm-1.2}   follow immediately  from Proposition~\ref{prop-3.4}.

\item
Statements (iii) and (iv) of Theorem~\ref{thm-1.2}   follow immediately  from Proposition~\ref{prop-3.8}.

\item
Further, as $\ell (x,y)\le h^{\frac{2}{3}-\delta}$ we can take $T^*=h^{\frac{1}{3}-\delta}$ and apply the same arguments.

\item
Finally, Statement (v) of Theorem~\ref{thm-1.2} follows from (\ref{eqn-2.44}) and (\ref{eqn-2.28}).
\end{enumerate}
\end{proof}

\chapter{Generalizations and final remarks}
\label{sect-4}

\begin{remark}\label{remark-4.1}
\begin{enumerate}[wide,label=(\roman*),  labelindent=0pt]
\item
In any dimension $d\ge 1$  without assumption  (\ref{eqn-1.3}) the simple rescaling with the scaling function 
\begin{gather}
\gamma_x=(\epsilon |V(x)-\tau|+h^{\frac{2}{3}} )
\label{eqn-4.1}
\intertext{results in the estimate}
|e_h(x,y,\tau)-e^\W_h (x,y,\tau)|\le Ch^{1-d}\gamma _x ^{(d-3)/2}
\label{eqn-4.2}
\end{gather}
as long as $\ell(x,y) \le  \gamma_x $ (and then   $\gamma_y\asymp \gamma_x$)\,\footnote{\label{foot-6} As $d=1$ Weyl expression should be modified so it contains a true eikonal  rather than  $\langle x-y,\theta\rangle - a(\frac{1}{2} (x+y),\theta)$ approximation.}.

\item
To explore the case $\ell(x,y) \ge \gamma_x$ observe that the standard Tauberian method, Part I, implies
\begin{gather}
e_h(x,x,\tau+h \gamma_x^{-\frac{1}{2}})   -e_h(x,x,\tau)\le Ch^{1-d}\gamma_x  ^{(d-3)/2} 
\label{eqn-4.3}\\
\shortintertext{Then}
|e_h(x,y,\tau+h \gamma_*^{-\frac{1}{2}} )  -e_h(x,y,\tau)|\le Ch^{1-d}\gamma ^{*\, (d-1)/4} \gamma_*^{(d-5)/4}
\label{eqn-4.4}
\end{gather}
with $\gamma^*=\max (\gamma_x,\,\gamma_y)$ and $\gamma_*=\min (\gamma_x,\,\gamma_y)$.

Applying standard Tauberian methods, Part II, we arrive to the estimate
\begin{multline}
|e_h(x,y,\tau)| \le CT_{x,y}^{-1} \gamma_* ^{\frac{1}{2}} \times h^{1-d}\gamma ^{*\, (d-1)/4} \gamma_*^{(d-5)/4} \\
\le C h^{1-d}\gamma ^{*\, (d-3)/4} \gamma_*^{(d-3)/4}=C h^{1-d}\gamma_x ^{(d-3)/4} \gamma_y^{(d-3)/4}
\label{eqn-4.5}
\end{multline}
where $T_{x,y}\gtrsim |x-y|^{\frac{1}{2}} \ge \gamma^{*\,\frac{1}{2}}$ is the minimal propagation time between $x$ and $y$ on the energy level $\tau$.
\end{enumerate}
Therefore, in the worst case remainder estimate is $O(h^{-\frac{2}{3}})$ for $d=1$, $O(h^{-\frac{4}{3}})$ for $d=2$ and $O(h^{1-d})$ for $d\ge 3$.
\end{remark}

\begin{remark}\label{remark-4.2}
\begin{enumerate}[wide,label=(\roman*),  labelindent=0pt]
\item
Under assumption  (\ref{eqn-1.3}) following arguments of Proposition~\ref{prop-3.2} we upgrade (\ref{eqn-4.3})  to 
\begin{gather}
e_h(x,x,\tau+h)   -e_h(x,x,\tau)\le Ch^{1-d}\gamma_x  ^{(d-2)/2} 
\label{eqn-4.6}\\
\shortintertext{and then}
|e_h(x,y,\tau)-e^\T_{\epsilon_1,h} (x,y,\tau)|\le Ch^{1-d}\gamma_x  ^{(d-2)/4}\gamma_y  ^{(d-2)/4}.
\label{eqn-4.7}
\end{gather}
Therefore, in the worst case remainder estimate is $O(h^{-\frac{1}{3}})$ for $d=1$, $O(h^{-1})$ for $d=2$ and $O(h^{1-d})$ for $d\ge 3$.

\item
However to replace $e^\T_{\epsilon_1,h} (x,y,\tau)$ by  $e^\W_{h} (x,y,\tau)$ and preserve this remainder estimate  we may need to add a correction term (in any dimension!) which in the regular zone is of magnitude  $h^{(1-d)/2}\ell(x,y)^{-(d+ 3)/4}$ and in the singular zone is $O( h^{(2-3d)/6}\ell(x,y)^{-(d+2)/4})$\, \footnote{\label{foot-7} As $d=1$ there is no shadow zone but there still is a singular zone where $\gamma_x\not\asymp\gamma_y$.}. For toy-model (\ref{eqn-2.1}) it is
\begin{multline}
e_{\corr, h}(x,y, \tau)=\frac{2}{(2\pi)^{\frac{d}{2}}\Gamma(\frac{d}{2}+1) h^{d+1}} \int_{-\infty}^\tau \int_0^\infty  (\tau-\tau')^{\frac{d}{2}} \\
\times 
 \cos\Bigl[h^{-1}\Bigl(\beta (x_1+y_1+2\tau) +\frac{1}{4\beta}  \bigl((x_1-y_1)^2 +|z'|^2\bigr)-\frac{1}{3}\beta^3\bigr) \Bigr) -\frac{\pi}{4}(d+2)\Bigr]\,d\beta d\tau\\
 -e^\W_h(x,y,\tau)
 \label{eqn-4.8}
\end{multline}
with $z'=x'-y'$. Then $e_{\corr,h} (x,x,\tau)$ can be expressed through  Airy function. 
\end{enumerate}
\end{remark}

\begin{remark}\label{remark-4.3}
To get rid off condition (\ref{eqn-1.3}) we can use  the simple rescaling with the scaling function
\begin{gather}
\gamma_{2x} =(\epsilon |V(x)-\tau| + |\nabla V|^2 +h )^{\frac{1}{2}}; 
\label{eqn-4.9}
\end{gather}
then $\hbar =\gamma_{2x}^{-2}h$ and the remainder estimates are not spoiled in dimensions $d\ge 2$, but in dimension $d=1$ it becomes 
$O(\hbar^{-\frac{1}{3}}\gamma_{2x}^{-1})= O(h^{-\frac{1}{3}}\gamma_{2x}^{-\frac{1}{3}})$ which is $O(h^{-\frac{1}{2}})$ in the worst case.
\end{remark}

\bibliographystyle{amsplain}

\end{document}